\documentclass[12pt, twoside]{amsart}
\usepackage{amssymb,amsmath,amscd,enumerate,verbatim}
\usepackage[colorlinks=true,linkcolor=blue,citecolor=blue]{hyperref}

\usepackage{graphicx}
\usepackage[all]{xy}
\usepackage{tikz}
\usetikzlibrary{matrix, arrows}

\usepackage{fullpage}

\numberwithin{equation}{section}

\newtheorem{theorem}{Theorem}[section]
\newtheorem{lemma}[theorem]{Lemma}
\newtheorem{proposition}[theorem]{Proposition}
\newtheorem{corollary}[theorem]{Corollary}

\theoremstyle{definition}
\newtheorem{definition}[theorem]{Definition}
\newtheorem{def-prop}[theorem]{Definition-Proposition}

\newtheorem{remark}[theorem]{Remark}
\newtheorem{example}[theorem]{Example}

\newtheorem{notation}[theorem]{Notation}

\DeclareMathOperator{\reg}{reg}
\DeclareMathOperator{\depth}{depth}

\DeclareMathOperator{\pd}{pd}

\newcommand{\PP}{{\mathbb P}}
\newcommand{\ZZ}{{\mathbb Z}}
\newcommand{\NN}{{\mathbb N}}

\def\A{{\mathcal A}}

\def\mm{{\mathfrak m}}

\def\a{{\bf a}}

\def\e{{\bf e}}

\def\s{{\bf s}}

\def\y{{\bf y}}

\def\tH{{\tilde{H}}}

\def\1{{\bf 1}}
\def\0{{\bf 0}}

\begin{document}

\title{Algebraic properties of toric rings of graphs}

\author{Huy T\`ai H\`a}
\address{Tulane University \\ Department of Mathematics \\
6823 St. Charles Ave. \\ New Orleans, LA 70118, USA}
\email{tha@tulane.edu}
\urladdr{http://www.math.tulane.edu/$\sim$tai/}

\author{Selvi Kara}
\address{University of South Alabama\\ Department of Mathematics and Statistics\\
411 University Blvd. North \\ Mobile, AL 36688, USA}
\email{selvi@southalabama.edu}
\urladdr{}

\author{Augustine O'Keefe}
\address{Connecticut College \\ Mathematics Department\\
 270 Mohegan Avenue Pkwy. \\ New London, CT 06320, USA}
\email{aokeefe@conncoll.edu}
\urladdr{}

\subjclass[2010]{}

\begin{abstract}
Let $G = (V,E)$ be a simple graph. We investigate the Cohen-Macaulayness and algebraic invariants, such as the Castelnuovo-Mumford regularity and the projective dimension, of the toric ring $k[G]$ via those of toric rings associated to induced subgraphs of $G$.
\end{abstract}

\maketitle


\section{Introduction}

Let $G = (V,E)$ be a simple graph over the vertex set $V$ and with edge set $E \subseteq 2^V$. Let $k$ be an arbitrary field, and identify the vertices and edges of $G$ with the variables in polynomial rings $k[V]$ and $k[E]$, respectively. The \emph{toric ring} associated to $G$, denoted by $k[G]$, is defined to be the image of the following monomial $k$-algebra homomorphism:
$$\phi_G: k[E] \longrightarrow k[V] \quad \text{given by} \quad e = \{x,y\} \in E \mapsto xy \in k[V].$$

Toric rings in general (defined to be the image of monomial maps between polynomial rings; see \cite{CS}) are the object of study in various areas in mathematics. Toric rings associated to graphs have attracted significant attention in combinatorial commutative algebra. For instance, their algebraic properties and invariants have been investigated in \cite{AH, BM, BOVT, CN, GV, HHKO, HMO, OH, OHH, Vill}. Their \emph{toric ideals} (the kernel of $\phi_G$) have also been studied in \cite{DA, DG, RTT, TT}.

In this paper, we examine the Cohen-Macaulayness and algebraic invariants of toric rings associated to graphs. Our approach is to see how properties and invariants of $k[G]$ could be derived from or bounded by similar properties and invariants of $k[H]$ for a subgraph $H$ of $G$. Our work hinges on the following observation: if $H$ is an \emph{induced} subgraph of $G$ then the homology groups of $k[H]$ are contained in the homology groups of $k[G]$ (see Lemma \ref{lem.retract}). Particularly, it follows that important algebraic invariants of $k[G]$ are bounded below by that of $k[H]$ and, under certain conditions, the non-Cohen-Macaulayness of $k[H]$ implies that of $k[G]$ (see Theorem \ref{thm.reduction}).

This approach allows us to quickly recover a main result of recent work of Biermann, O'Keefe and Van Tuyl (\cite[Theorem 2.6]{BOVT} and, subsequently, \cite[Theorem 1.1]{BOVT}) for the regularity of $k[G]$. We also obtain similar statements for the projective dimension of $k[G]$. Specifically, we prove the following theorems.

\noindent{\bf Theorem \ref{thm.BOVT}.} Let $G$ be a simple graph. Suppose that $G$ contains an induced subgraph which is the disjoint union of graphs $H_1, \dots, H_s$. Then
\begin{enumerate}
\item $\reg k[G] \ge \sum_{i=1}^s \reg k[H_i].$
\item $\pd k[G] \ge \sum_{i=1}^s \pd k[H_i].$
\end{enumerate}

\noindent{\bf Theorem \ref{thm.BOVT1}.} Let $G$ be a simple graph. Suppose that $G$ contains an induced subgraph which is the disjoint union of complete bipartite graphs $K_{m_1,n_1}, \dots, K_{m_s,n_s}$. Then
\begin{enumerate}
\item $\reg k[G] \ge \sum_{i=1}^s \min\{m_i,n_i\} - s.$
\item $\pd k[G] \ge \sum_{i=1}^s (m_i-1)(n_i-1).$
\end{enumerate}

To prove Theorems \ref{thm.BOVT} and \ref{thm.BOVT1}, we make use of our initial observation that the regularity and projective dimension of $k[G]$ are bounded below by that of $k[H]$ for an induced subgraph $H$ of $G$ (Theorem \ref{thm.reduction}), and show that the regularity and the projective dimension of toric rings are additive with respect to disjoint union of graphs (Lemma \ref{lem.disjoint}).

Our method further leads us to the problem of finding ``forbidden'' structures in $G$ which prevent $k[G]$ from being Cohen-Macaulay. We give such a forbidden structure in the following theorem.

\noindent {\bf Theorem \ref{thm.nCM}.} Let $G = (V,E)$ be a simple graph. Suppose that $|E| \le |V|+2$ and $G$ contains an induced subgraph which consists of:
\begin{itemize}
\item two vertex-disjoint odd cycles; and
\item two vertex-disjoint (except possibly at their endpoint vertices) paths of length $\ge 2$ connecting these cycles.
\end{itemize}
Then the toric ring $k[G]$ is not Cohen-Macaulay.

To prove Theorem \ref{thm.nCM}, we also make use of our initial observation, Theorem \ref{thm.reduction}, to reduce to the case where $G$ consists of exactly two odd cycles that are connected by exactly two paths of length $\ge 2$ which do not have any vertex in common (except possibly at the endpoints). In this case, we then apply a well-known formula (cf. Theorem \ref{thm.homology}) which relates the graded Betti numbers of $k[G]$ to the reduced homology of certain simplicial complexes and, by a suitable choice of $\s \in \ZZ^V$, show that an appropriate graded Betti number of $k[G]$ is not zero.

It follows from the main theorem of Ohsugi and Hibi \cite{OH} that if every pair of induced odd cycles in $G$ either share a vertex or are connected by an edge then $k[G]$ is normal. By a celebrated result of Hochster \cite{H}, $k[G]$ is Cohen-Macaulay in this case. On the other hand, if $G$ consists of exactly two odd cycles which are connected by only one path (of length $\ge 2$) then $k[G]$ is not normal but Cohen-Macaulay. Thus, the structure of having two induced odd cycles connected by (at least) two paths of length $\ge 2$ is, in some sense, the minimal structure that one could look for to break the Cohen-Macaulayness of $k[G]$. 


The paper is structured as follows. In the next section, we establish notation and terminology used in the paper. In Section \ref{sec.reg}, we give our initial observation that properties and invariants of $k[G]$ could be derived from or bounded by those of $k[H]$ for an induced subgraph $H$ of $G$. We will also prove our first two main results, Theorems \ref{thm.BOVT} and \ref{thm.BOVT1}, in this section. In Section \ref{sec.Delta}, we focus on a forbidden structure and, for a suitable choice of $\s \in \ZZ^V$, completely describe the associated simplicial complex $\Delta_\s$. This description of $\Delta_\s$ is key in the proof of the sufficient condition for the non-Cohen-Macaulayness of $k[G]$ presented in the last section. In Section \ref{sec.nCM}, we prove our last main result, Theorems \ref{thm.nCM}, and give a number of examples.

\noindent{\bf Acknowledgement.} The first named author is partially supported by the Simons Foundation (grant \#279786). Part of this work was done when the first two authors were visiting Vietnam Institute for Advanced Study in Mathematics (VIASM). The authors would like to thank VIASM for its support and hospitality. The authors would also like to thank K. Kimura for providing Example \ref{ex.Kimura} and for spotting a gap in our first draft.


\section{Preliminaries} \label{sec.prel}

We assume that the reader is familiar with basic concepts and terminology in combinatorial commutative algebra. For unexplained notations, we refer the reader to standard texts in the area \cite{BS,D,E}.

Throughout the paper, $G = (V,E)$ shall denote a \emph{simple} graph over the vertex set $V$ and edge set $E \subseteq 2^V$ (where $|V| = n$ and $|E| = m$). A simple graph is a graph without loops nor multiple edges. When the vertex set and edge set of $G$ are not specified, we shall use $V_G$ and $E_G$, respectively, to denote these sets.

The following familiar structures in graphs will be used in our discussion.

\begin{definition} Let $G$ be a simple graph.
\begin{enumerate}
\item A graph $H$ is a \emph{subgraph} of $G$ if $V_H \subseteq V_G$ and $E_H \subseteq E_G$.
\item A subgraph $H$ of $G$ is an \emph{induced} subgraph if for any $x,y \in V_H$, $e = \{x,y\} \in E_H$ if and only if $e \in E_G$.
\item A \emph{path} in $G$ is an alternating sequence of distinct vertices and edges (except possibly the endpoints) $x_1, e_1, x_2, e_2, \dots, e_l, x_{l+1}$, where $e_i = \{x_i, x_{i+1}\} \in E_G$.
\item A \emph{cycle} is a path whose endpoint vertices coincide. An \emph{odd} cycle is a cycle with odd number of edges.
\end{enumerate}
\end{definition}

Let $k$ be a field, and let $k[V]$ and $k[E]$ be polynomial rings, where by abusing notation we identify the vertices and edges of $G$ with indeterminates. Let $k[G]$ denote the image of the $k$-algebra homomorphism:
$$\phi_G: k[E] \longrightarrow k[V] \quad \text{given by} \quad e = \{x,y\} \in E \mapsto xy \in k[V].$$

\begin{definition} Let $G = (V,E)$ be a simple graph on $n$ vertices and $m$ edges. Let $M_G$ be the $n \times m$ incidence matrix of $G$. For $e \in E$, we shall denote by $\a_e$ the column of $M_G$ corresponding to $e$. Let $\A = \{\a_e ~|~ e \in E\}$ and let $S_G = \NN\A \subseteq \ZZ^n$ be the semigroup spanned by $\A$.
\end{definition}

\begin{remark} Suppose that $V = \{x_1, \dots, x_n\}$. The polynomial ring $k[V]$ has a natural $\NN^n$-graded structure in which $\deg(x_i) = \e_i$, where $\{\e_1, \dots, \e_n\}$ represents the standard unit basis for $\ZZ^n$. The monomial map $\phi_G$ induces an $\NN^n$-graded structure for $k[E]$. Moreover, since $\deg(xy) = \a_e$, when $e = \{x,y\} \in E$, $\phi_G$ in fact induces a natural $S_G$-graded structure for $k[E]$ and $k[G]$.
\end{remark}

For $\s \in \ZZ^n$, we shall denote by $\beta_{i,\s}^{k[E]}(k[G])$ (or simply $\beta_{i,\s}(k[G])$) the number of generators of multidegree $\s$ of the $i$th syzygy module of $k[G]$ (i.e., the $(i,\s)$-graded Betti number of $k[G]$.) Then $\beta_{i,\s}(k[G]) = 0$ if $\s \not\in S_G$.

The center of our work is to investigate the vanishing and non-vanishing of graded Betti numbers of $k[G]$. We shall recall a construction from \cite{AH}, which facilitates a mean to relate graded Betti numbers of $k[G]$ to reduced homology groups of certain simplicial complexes, the \emph{degree complexes}.

\begin{definition} Let $G = (V,E)$ be a simple graph, and let $\s \in \ZZ^V$. Define $\Delta_\s$ to be the simplicial complex on $E$ whose facets are (maximal sets) of the form
$$\{e \in E ~\big|~ c_e > 0\}, \text{ where } \s = \sum_{e \in E} c_e \a_e \text{ and } c_e \in \ZZ_{\ge 0} \ \forall \ e \in E.$$
\end{definition}
\noindent(Clearly, $\Delta_\s = \emptyset$ if $\s \not\in S_G$.)

\begin{theorem}[\protect{\cite[Lemma 4.1]{AH}}] \label{thm.homology}
$\beta_{i,\s}(k[G]) = \dim_k \tilde{H}_{i-1}(\Delta_\s;k).$
\end{theorem}


\section{Regularity and projective dimension} \label{sec.reg}

The aim of this section is to provide bounds for the regularity and projective dimension of the toric ring $k[G]$ in terms of the toric rings associated to induced subgraphs of $G$. Particularly, we shall recover a main result of Biermann, O'Keefe and Van Tuyl \cite{BOVT} on the regularity of $k[G]$ and prove similar statements for the projective dimension of $k[G]$.

We shall start by investigating how homology groups of $k[H]$, for an induced subgraph $H$ of $G$, are compared to those of $k[G]$. This is also a keystone of our study on the Cohen-Macaulayness of $k[G]$ later on in the paper. To achieve this, we will make use of the notion of retract algebras, which we shall now recall.

\begin{definition} Let $k$ be a field and let $S = \bigoplus_{n \ge 0}S_n$ be a graded ring. We call $S$ a \emph{standard} graded $k$-algebra if $S_0 = k$ and $S$ is generated by $S_1$ as a $k$-algebra. In this case, let $\mm_S = \bigoplus_{n \ge 1} S_n$ be its maximal homogeneous ideal.
\end{definition}

\begin{definition} \label{def.retract}
Let $S$ be a $k$-algebra and let $T$ be a $k$-subalgebra of $S$. The natural inclusion $\iota: T \hookrightarrow S$ is said to be an \emph{algebra retract} if there exists a surjective $k$-algebra homomorphism (a \emph{retraction map}) $\epsilon: S \twoheadrightarrow T$ such that $\epsilon \circ \iota = \text{id}_T$.
\end{definition}

\begin{notation} For a standard graded $k$-algebra $S$, we shall write $H_\bullet(S)$ for the Koszul homology of $S$ with respect to a system of minimal generators of $\mm_S$.
\end{notation}

The following observation is key for our work. This observation allows us to relate homological properties and algebraic invariants of $k[G]$ to that of toric rings associated to induced subgraphs of $G$.

\begin{lemma} \label{lem.retract}
Let $T \subseteq S$ be an algebra retract of standard graded $k$-algebras. Then the inclusion $T \hookrightarrow S$ induces the following algebra retracts
$$H_\bullet(T) \subseteq H_\bullet(S).$$
\end{lemma}

\begin{proof} This statement is the content of \cite[Proposition 2.4]{OHH}. We shall include the proof for completeness.

Let $\epsilon: S \twoheadrightarrow T$ be the retraction map. We may choose a $k$-basis $\y_S = y_1, \dots, y_p$ of $S_1$ such that for some $q \le p$, $\y_T = y_1, \dots, y_q$ forms a $k$-basis of $T_1$, and such that $\epsilon(y_j) = y_j$ for $j = 1, \dots, q$ and $\epsilon(y_j) = 0$ for $j = q+1, \dots, p$. Let $K(\y_S)$ and $K(\y_T)$ be the Koszul complexes of $S$ and $T$ with respect to $\y_S$ and $\y_T$. Then $H_\bullet(S) = H_\bullet(K(\y_S))$ and $H_\bullet(T) = H_\bullet(K(\y_T))$.
\[
\xymatrix{
&H_\bullet(K(\y_S)) \ar@{->>}[r]^{\epsilon_*} &H_\bullet(K(\y_T)) \ar@{{(}.>>}[ld]^{\text{id}}\\
&H_\bullet(K(\y_T)) \ar@{_{(}->}[u]^{\iota_*}}\\
\]

Observe that the inclusion $\iota: T \hookrightarrow S$ induces a homomorphism $\iota_*: H_\bullet(K(\y_T)) \rightarrow H_\bullet(K(\y_S))$, and the map $\epsilon: S \twoheadrightarrow T$ induces a homomorphism
$$\epsilon_*: H_\bullet(K(\y_S)) \rightarrow H_\bullet(K(\epsilon(\y_S))) \simeq H_\bullet(K(\y_T)) \otimes \left(\bigwedge W\right),$$
where $W$ is a $(p-q+1)$-dimensional $k$-vector space. Particularly, $H_\bullet(K(\y_T))$ is a subalgebra of $H_\bullet(K(\epsilon(\y_S)))$ and is precisely the image of $\epsilon_* \circ \iota_*$. This implies that $\epsilon_* \circ \iota_* = \text{id}_{H_\bullet(K(\y_T))}$, i.e., $H_\bullet(T) \subseteq H_\bullet(S)$ is an algebra retract.
\end{proof}

\begin{corollary} \label{cor.reduction}
Let $A \subseteq B$ be an algebra retract of graded $k$-algebras. Suppose that $A = T/I$ and $B = S/J$, where $T$ and $S$ are standard graded polynomial rings, and $I \subseteq T$ and $J \subseteq S$ are homogeneous ideals containing no forms of degree 1. Then for any $i,j \in \ZZ$, we have
$$\beta_{ij}^T(A) \le \beta_{ij}^S(B).$$
\end{corollary}

\begin{proof} The statement follows from Lemma \ref{lem.retract} and the fact that graded Betti numbers can be computed from homology of the Koszul complexes.
\end{proof}

\begin{theorem} \label{thm.reduction}
Let $G$ be a simple graph and let $H$ be an induced subgraph of $G$.
\begin{enumerate}
\item $\reg k[H] \le \reg k[G]$ and $\pd k[H] \le \pd k[G]$.
\item If $|E_H| - |V_H| \ge |E_G| - |V_G|$ and $k[H]$ is not Cohen-Macaulay then neither is $k[G]$.
\end{enumerate}
\end{theorem}

\begin{proof} Observe that the natural inclusion $k[H] \hookrightarrow k[G]$ is an algebra retract with retraction map $\epsilon: k[G] \rightarrow k[H]$ defined as follows: for any edge $e = \{x,y\} \in G$,
$$k[G] \ni xy \mapsto \left\{ \begin{array}{lcl} xy & \text{if} & \{x,y\} \in H \\ 0 & \text{if} & \{x,y\} \not\in H. \end{array} \right.$$
Thus, (1) follows directly from Corollary \ref{cor.reduction}.

To prove (2), assume that $k[H]$ is not Cohen-Macaulay. It follows, by Hochster's work \cite{H}, that $k[H]$ is not normal. The main theorem of Ohsugi and Hibi \cite{OH} then implies that $H$ is not bipartite (and, in particular, $G$ is not bipartite). By \cite[Proposition 3.2]{Vill}, we have $\dim k[H] = |V_H|$ and $\dim k[G] = |V_G|$. Since $k[H]$ is not Cohen-Macaulay, we have $\depth k[H] < \dim k[H] = |V_H|$. Thus, by the Auslander-Buchsbaum formula, we have
$\pd k[H] > |E_H| - |V_H|.$
This, together with (1), implies that
$$\pd k[G] > |E_H| - |V_H| \ge |E_G| - |V_G|.$$
By Auslander-Buchsbaum formula again, we have
$$\depth k[G] < |V_G| = \dim k[G].$$
Hence, $k[G]$ is not Cohen-Macaulay, and (2) is proved.
\end{proof}

We are now ready to state our first main theorem, the first part of which recovers \cite[Theorem 2.6]{BOVT}.

\begin{theorem}[\protect{See \cite[Theorem 2.6]{BOVT}}] \label{thm.BOVT}
Let $G$ be a simple graph. Suppose that $G$ contains an induced subgraph which is the disjoint union of graphs $H_1, \dots, H_s$. Then
\begin{enumerate}
\item $\reg k[G] \ge \sum_{i=1}^s \reg k[H_i].$
\item $\pd k[G] \ge \sum_{i=1}^s \pd k[H_i].$
\end{enumerate}
\end{theorem}

\begin{proof} Let $H$ be the disjoint union of $H_1, \dots, H_s$. Then $H$ is an induced subgraph of $G$. By Theorem \ref{thm.reduction}, we have
\begin{itemize}
\item $\reg k[G] \ge \reg k[H]$; and
\item $\pd k[G] \ge \pd k[H].$
\end{itemize}
The conclusion now follows from Lemma \ref{lem.disjoint} below.
\end{proof}

\begin{lemma} \label{lem.disjoint}
Let $H$ be the disjoint union of simple graphs $H_1, \dots, H_s$. Then
\begin{enumerate}
\item $\reg k[H] = \sum_{i=1}^s \reg k[H_i]$.
\item $\pd k[H] = \sum_{i=1}^s \pd k[H_i]$.
\end{enumerate}
\end{lemma}

\begin{proof} Suppose that $H = (V,E)$ and $H_i = (V_i,E_i)$ for $i = 1, \dots, s$. It is easy to see that
$$k[E] = k[E_1] \otimes_k \dots \otimes_k k[E_s] \text{ and } k[H] = k[H_1] \otimes_k \dots \otimes k[H_s].$$
Thus, the minimal free resolution of $k[H]$ as a $k[E]$-module is obtained by taking the tensor product of those of $k[H_i]$ (as a $k[E_i]$-module). In particular, we have
$$\beta_{t, \s}(k[H]) = \sum_{\substack{t_1 + \dots + t_s = t \\ \s_1 + \dots + \s_s = \s}} \beta_{t_i,\s_i}(k[H_i]).$$
The assertion now follows by the definition of regularity and projective dimension.
\end{proof}

As a consequence of Theorem \ref{thm.BOVT}, we also recover \cite[Theorem 1.1]{BOVT} and prove a similar statement for the projective dimension. Recall that a \emph{complete bipartite} graph $K_{u,v}$ is the graph consisting of two disjoint subsets of the vertices $X$ and $Y$, where $|X| = u$ and $|Y|=v$, and edges $\{xy ~|~ x \in X, y \in Y\}$.

\begin{theorem}[\protect{See \cite[Theorem 1.1]{BOVT}}] \label{thm.BOVT1}
Let $G$ be a simple graph. Suppose that $G$ contains an induced subgraph which is the disjoint union of complete bipartite graphs $K_{m_1,n_1}, \dots, K_{m_s,n_s}$. Then
\begin{enumerate}
\item $\reg k[G] \ge \sum_{i=1}^s \min\{m_i,n_i\} - s.$
\item $\pd k[G] \ge \sum_{i=1}^s (m_i-1)(n_i-1).$
\end{enumerate}
\end{theorem}

\begin{proof} The conclusion follows from Theorem \ref{thm.BOVT} and Lemma \ref{lem.Kst} below.
\end{proof}

Part (1) of the following lemma is a direct consequence of \cite[Lemma 3.1]{CM}. It was also proved by a different method in \cite{BOVT}.

\begin{lemma} \label{lem.Kst}
Let $K_{u,v}$ be a complete bipartite graph. Then
\begin{enumerate}
\item $\reg k[K_{u,v}] = \min\{u-1,v-1\}.$
\item $\pd k[K_{u,v}] = (u-1)(v-1).$
\end{enumerate}
\end{lemma}

\begin{proof} It is easy to see that $k[K_{u,v}]$ can be viewed as the coordinate ring of the Segre embedding $\PP^{u-1} \times \PP^{v-1} \longrightarrow \PP^{uv-1}$. Part (1) now follows from \cite[Lemma 3.1]{CM}.

To prove part (2), observe that $k[K_{u,v}]$ is Cohen-Macaulay (cf. \cite{SV}). Thus, by \cite[Proposition 3.2]{Vill}, we have $\pd k[K_{u,v}] = uv - \dim k[K_{u,v}] = uv - (u+v-1) = (u-1)(v-1).$
\end{proof}


\section{A forbidden structure and its degree complex} \label{sec.Delta}

Our preliminary result, Theorem \ref{thm.reduction}, allows us to focus on structures that prevent $k[G]$ from being Cohen-Macaulay. In this section, we shall consider such a forbidden structure and, for a suitable multidegree $\s$, describe its degree complex $\Delta_\s$. This degree complex allows us to compute certain graded Betti number of $k[G]$, based on Theorem \ref{thm.homology}, and to conclude that $k[G]$ is not Cohen-Macaulay in this case.

Our forbidden structure is a graph which consists of exactly two induced odd cycles $C_1$ and $C_2$, which are connected by two vertex-disjoint (except possibly at their endpoints) paths $P_1$ and $P_2$ of length $\ge 2$, as depicted in Figure \ref{fig.forbidden}. Throughout this section, we shall assume that $G$ is such a graph.

\begin{notation} \label{not.basic}
Throughout this section, we shall label the vertices and edges of $G$ in the following more convenient way:
\begin{itemize}
\item $V_{C_1} = \{x_1, \dots, x_{2l+1}\}$, $E_{C_1} = \{e_1, \dots, e_{2l+1}\}$, where $e_j = \{x_j, x_{j+1}\}$.
\item $V_{C_2} = \{y_1, \dots, y_{2l'+1}\}$, $E_{C_2} = \{e_1', \dots, e_{2l'+1}'\}$, where $e_j' = \{y_j, y_{j+1}\}$.
\item $V_{P_1} = \{z_0, \dots, z_p\}$, $E_{P_1} = \{f_1, \dots, f_p\}$, where $f_j = \{z_{j-1}, z_j\}$, $z_0 \in V_{C_1}$ and $z_p \in V_{C_2}$.
\item $V_{P_2} = \{w_0, \dots, w_q\}$, $E_{P_2} = \{f_1', \dots, f_q'\}$, where $f_j' = \{w_{j-1}, w_j\}$, $w_0 \in V_{C_1}$ and $w_q \in V_{C_2}$.
\end{itemize}

\begin{figure}[h!]
\centering
\includegraphics[width=\textwidth]{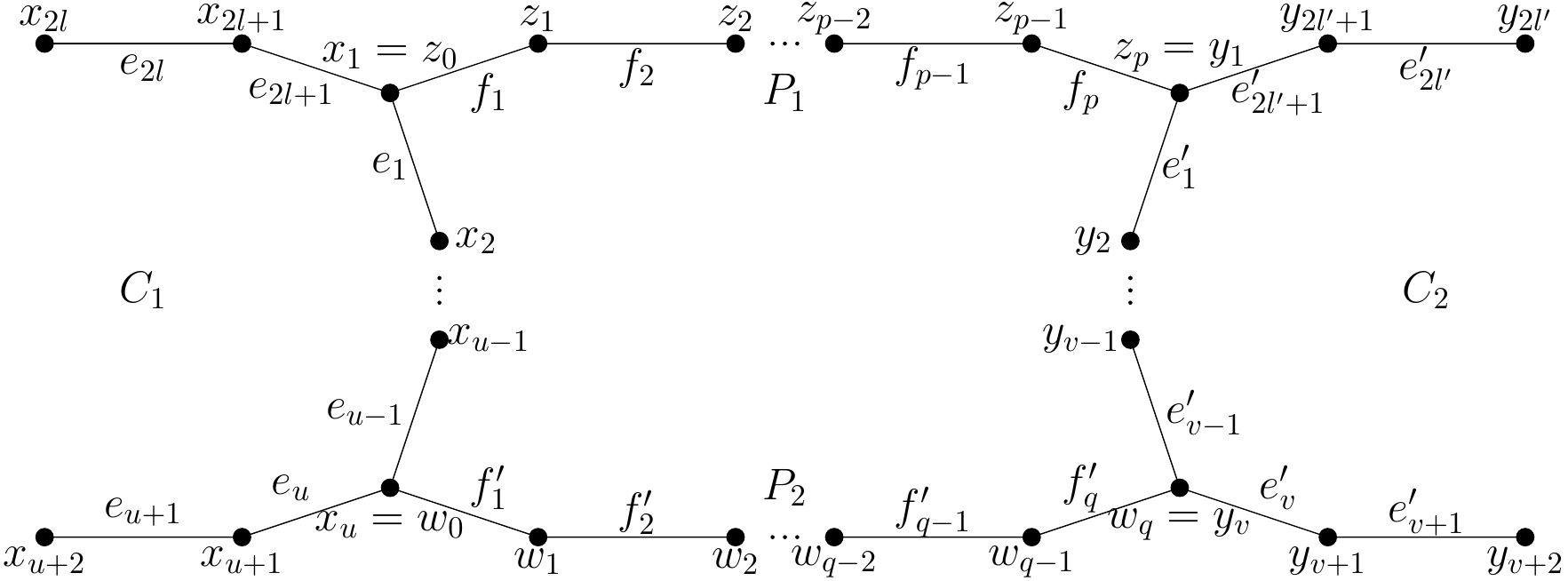}
\centering\caption{A forbidden structure.}\label{fig.forbidden}
\end{figure}
\end{notation}

Our choice of $\s = (s_x ~|~ x \in V) \in \ZZ^V$ is given as follows:
$$s_x = 1 + \big| \{j ~|~ x \in P_j\} \big|.$$

\begin{remark} \label{rmk.coefficient}
It can be seen that in the decomposition $\s=\sum_{e \in E} c_e \a_e$, the coefficients $c_{f_1}, \dots, c_{f_p}$ on path $P_1$ are completely determined given $c_{f_1}$. More specifically, since $s_{z_1} = 2$, we must have $c_{f_1} \le 2$. If $c_{f_1} = 2$ then it forces $c_{f_2} = 0$. This, together with $s_{z_2} = 2$, again forces $c_{f_3} = 2$. Keep going in this fashion, we have $c_{f_j} = 2$ if $j$ is odd and $c_{f_j} = 0$ if $j$ is even. The situation is similar if we start with $c_{f_1} = 0$. If $c_{f_1} = 1$, then since $s_{z_1} = 2$, we must have $c_{f_2} = 1$. Again, since $s_{z_2} = 2$, we must have $c_{f_3} = 1$. Keep going in this fashion, we have $c_{f_j} = 1$ for all $1 \le j \le p$. The same observation also works for the coefficients $c_{f_1'}, \dots, c_{f_q'}$ on path $P_2$.

It can be further seen that for a vertex $x_i \not\equiv z_0, w_0$ on $C_1$, since $s_{x_i} = 1$, we must have either $c_{e_{i-1}} = 0$ and $c_{e_i} = 1$ or $c_{e_{i-1}} = 1$ and $c_{e_i} = 0$ (i.e., knowing one of the coefficients $\{c_{e_{i-1}}, c_{e_i}\}$ determines the other one.) The exception to this rule is when $x_i \equiv z_0$ or $x_i \equiv w_0$, where the determination of $c_{e_{i-1}}$ and $c_{e_i}$ also depends on the value of $c_{f_1}$ or $c_{f_1'}$. The same observation also works for the coefficients of edges on $C_2$.
\end{remark}

\begin{proposition} \label{lem.facets}
$\Delta_\s$ has exactly 4 facets which can be explicitly described.
\end{proposition}

\begin{proof} Consider an expression $\s = \sum_{e \in E} c_e \a_e$. Our argument is a case by case analysis. We consider the following cases depending on the endpoints of $P_1$ and $P_2$, which in turn have their subcases depending on the value of $c_{f_1}$.

\noindent{\bf Case 1:} $P_1$ and $P_2$ share both endpoint vertices. Without loss of generality, suppose that $z_0 = w_0 = x_1$ and $z_p = w_q = y_1$. Since $s_{z_1} = s_{w_1} = 2$, we must have $c_{f_1}, c_{f_1'} \le 2$.

\noindent{\it Case 1a:} $c_{f_1} = 2$. It follows from  Remark \ref{rmk.coefficient} that, in this case, $c_{f_j} = 2$ if $j$ is odd and $c_{f_j} = 0$ if $j$ is even. Observe that since $s_{x_1} = 3$, we cannot have both $c_{e_1} = 1$ and $c_{e_{2l+1}} = 1$ (and $c_{f_1} = 2$). Thus, by Remark \ref{rmk.coefficient}, we must have $c_{e_j} = 0$ for $j$ odd and $c_{e_j} = 1$ for $j$ even. This, together with the fact that $s_{x_1} = 3$, forces $c_{f_1'} = 1$. Again, by Remark \ref{rmk.coefficient}, we deduce that $c_{f_j'} = 1$ for all $j = 1, \dots, q$.

If $p$ is odd then since $c_{f_p} = 2$, $c_{f_q'} = 1$ and $s_{y_1} = 3$, we must have $c_{e_1'} = c_{e_{2l'+1}'} = 0$. This implies, by Remark \ref{rmk.coefficient}, that $c_{e_j'} = 0$ if $j$ is odd and $c_{e_j'} = 1$ if $j$ is even. Thus, the expression $\s = \sum_{e \in E} c_e\a_e$ is uniquely determined, which gives the following facet of $\Delta_\s$:
\begin{align*}
G_{11} = & \{e_j ~|~ 1 \le j \text{ even } \le 2l+1 \} \cup \{f_j ~|~ 1 \le j \text{ odd } \le p\} \\
& \cup \{f_j' ~|~ 1\le j \le q\} \cup \{e_j' ~|~ 1 \le j \text{ even } \leq 2l'+1\}.
\end{align*}
\begin{figure}[h!]
\centering
\includegraphics[height=2in]{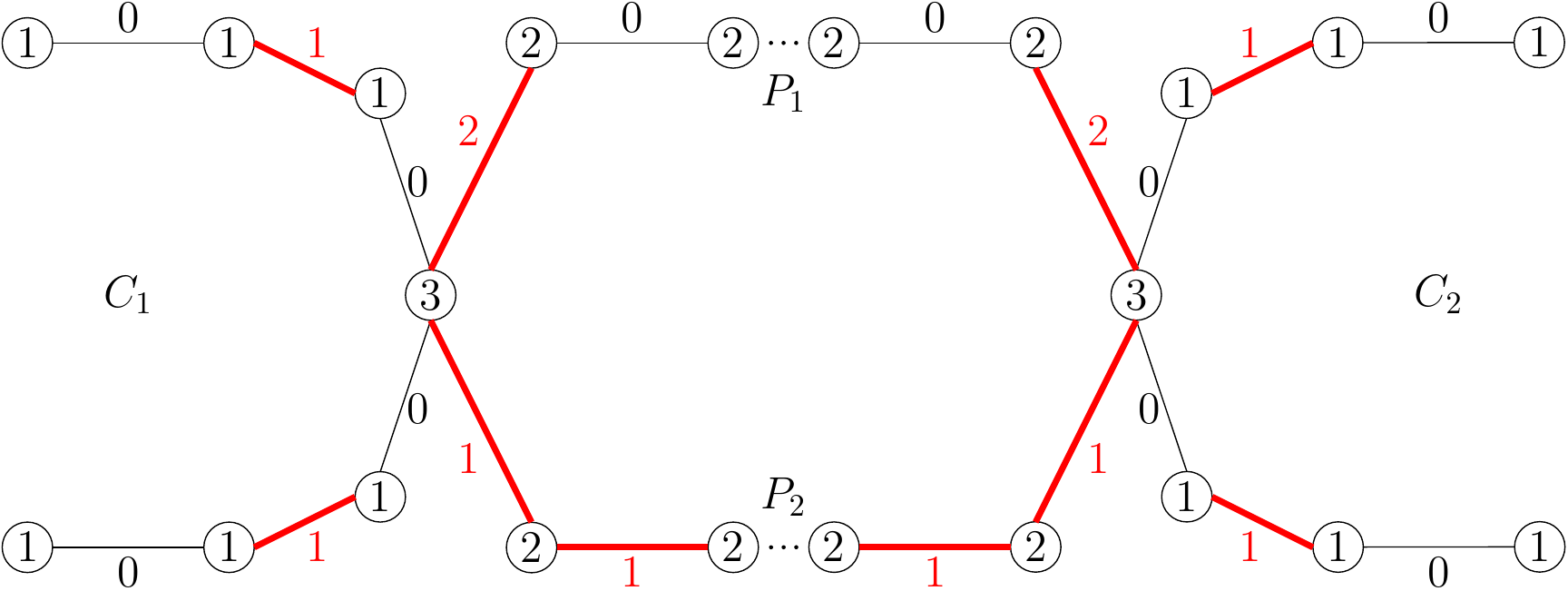}
\centering\caption{A facet of $\Delta_\s$: Case 1a when $p$ is odd.}\label{fig.1a}
\end{figure}

If $p$ is even then since $c_{f_p} = 0$, $c_{f_q'} = 1$ and $s_{y_1} = 3$, we must have $c_{e_1'} = c_{e_{2l'+1}'} = 1$. This implies, by Remark \ref{rmk.coefficient}, that $c_{e_j'} = 1$ if $j$ is odd and $c_{e_j'} = 0$ if $j$ is even. Thus, the expression $\s = \sum_{e \in E} c_e\a_e$, in this case, gives the following facet of $\Delta_\s$ instead:
\begin{align*}
G_{11}' = & \{e_j ~|~ 1 \le j \text{ even } \le 2l+1 \} \cup \{f_j ~|~ 1 \le j \text{ odd } \le p\} \\
& \cup \{f_j' ~|~ 1\le j \le q\} \cup \{e_j' ~|~ 1 \le j \text{ odd } \leq 2l'+1\}.
\end{align*}

\noindent{\it Case 1b:} $c_{f_1} = 1$. By Remark \ref{rmk.coefficient}, we have $c_{f_j} = 1$ for all $j = 1, \dots, p$. Suppose first that $c_{e_1} > 0$. Then $c_{e_1} = 1$, and it follows from Remark \ref{rmk.coefficient} that $c_{e_j} = 1$ if $j$ is odd and $c_{e_j} = 0$ if $j$ is even. Since $s_{x_1} = 3$ and $c_{f_1} = c_{e_1} = c_{e_{2l+1}} = 1$, we must have $c_{f_1'} = 0$. Remark \ref{rmk.coefficient} gives us that $c_{f_j'} = 0$ if $j$ is odd and $c_{f_j'} = 2$ if $j$ is even.

If $q$ is odd then since $c_{f_p} = 1$ and $c_{f_q'} = 0$, we must have $c_{e_1'} = c_{e_{2l'+1}'} = 1$. This implies, by Remark \ref{rmk.coefficient}, that $c_{e_j'} = 1$ if $j$ is odd and $c_{e_j'} = 0$ if $j$ is even, and eventually determines the following facet of $\Delta_\s$:
\begin{align*}
G_{21} = & \{e_j ~|~ 1 \le j \text{ odd } \le 2l+1 \} \cup \{f_j ~|~ 1 \le j \le p\} \\
& \cup \{f_j' ~|~ 1\le j \text{ even }\le q\} \cup \{e_j' ~|~ 1 \le j \text{ odd } \leq 2l'+1\}.
\end{align*}

\begin{figure}[h!]
\centering
\includegraphics[height=2in]{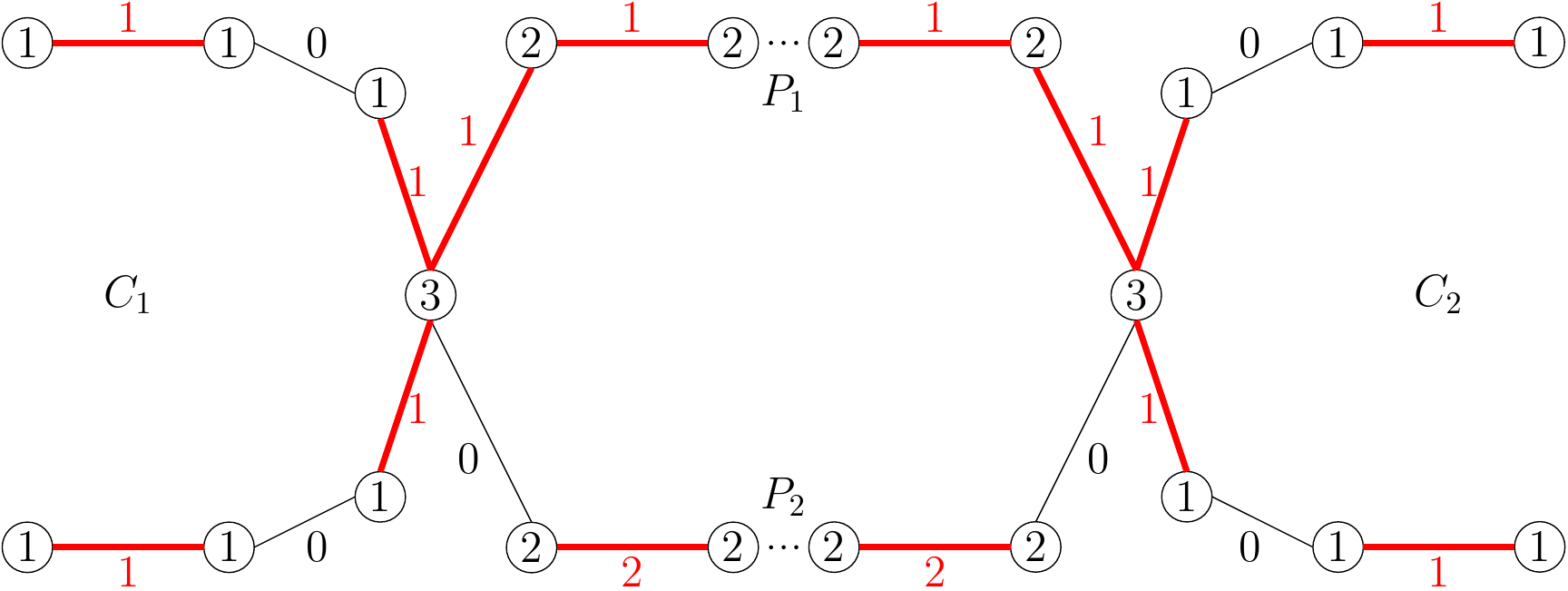}
\centering\caption{A facet of $\Delta_\s$: Case 1b when $c_{e_1} = 1$ and $q$ is odd.}\label{fig.1b1}
\end{figure}

If $q$ is even then similarly, we obtain the following facet of $\Delta_\s$ instead:
\begin{align*}
G_{21}' = & \{e_j ~|~ 1 \le j \text{ odd } \le 2l+1 \} \cup \{f_j ~|~ 1 \le j \le p\} \\
& \cup \{f_j' ~|~ 1\le j \text{ even }\le q\} \cup \{e_j' ~|~ 1 \le j \text{ even } \leq 2l'+1\}.
\end{align*}

Suppose, on the other hand, that $c_{e_1} = 0$. Then by Remark \ref{rmk.coefficient}, we have $c_{e_j} = 0$ if $j$ is odd and $c_{e_j} = 1$ if $j$ is even. This, together with the fact that $c_{f_1} = 1$ and $s_{x_1} = 3$, forces $c_{f_1'} = 2$. It follows, by Remark \ref{rmk.coefficient}, that $c_{f_j'} = 2$ if $j$ is odd and $c_{f_j'} = 0$ if $j$ is even.

Now, if $q$ is odd then, by a similar argument as above, we obtain the facet
\begin{align*}
G_{22} = & \{e_j ~|~ 1 \le j \text{ even } \le 2l+1 \} \cup \{f_j ~|~ 1 \le j \le p\} \\
& \cup \{f_j' ~|~ 1\le j \text{ odd }\le q\} \cup \{e_j' ~|~ 1 \le j \text{ even } \leq 2l'+1\},
\end{align*}
and if $q$ is even then we obtain the facet
\begin{align*}
G_{22}' = & \{e_j ~|~ 1 \le j \text{ even } \le 2l+1 \} \cup \{f_j ~|~ 1 \le j \le p\} \\
& \cup \{f_j' ~|~ 1\le j \text{ odd }\le q\} \cup \{e_j' ~|~ 1 \le j \text{ odd } \leq 2l'+1\}.
\end{align*}

\begin{figure}[h!]
\centering
\includegraphics[height=2in]{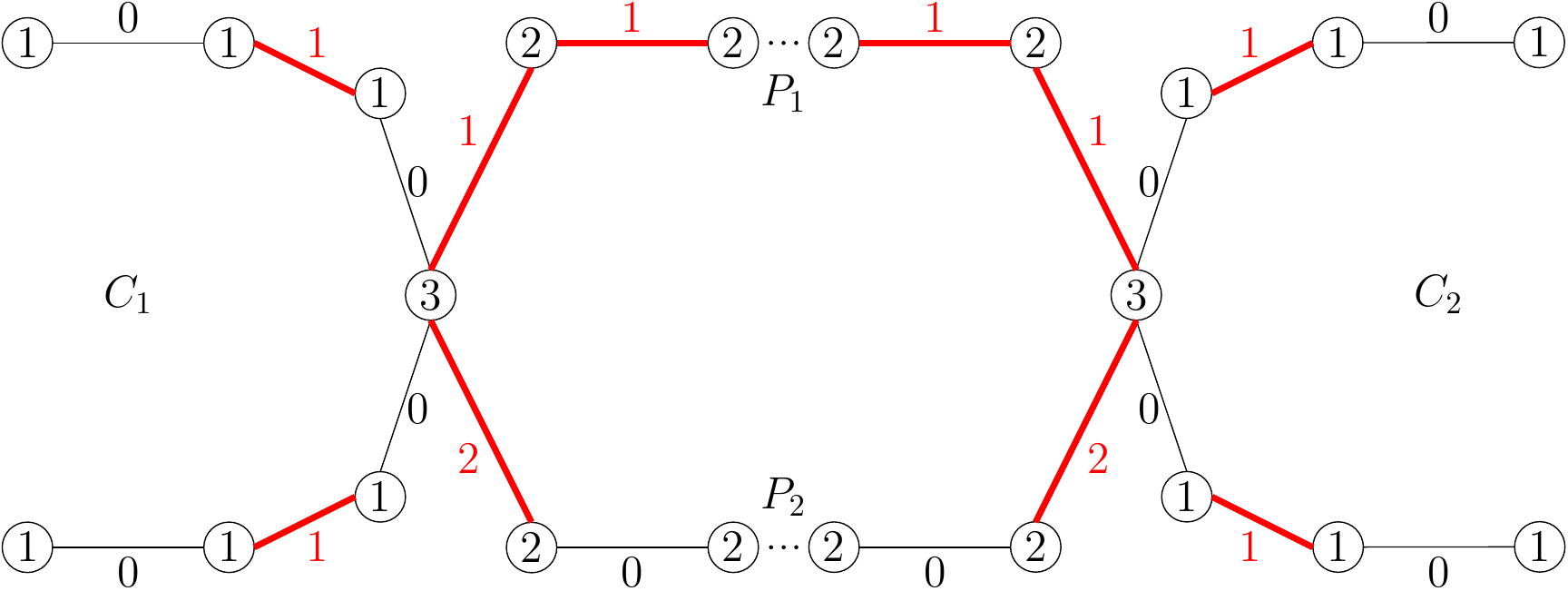}
\centering\caption{A facet of $\Delta_\s$: Case 1b when $c_{e_1} = 0$ and $q$ is odd.}\label{fig.1b0}
\end{figure}

\noindent{\it Case 1c:} $c_{f_1} = 0$. Observe if $c_{e_1} = 0$ then by tracing around $C_1$, we must have $c_{e_{2l+1}} = 0$. This implies that $c_{f_1'} = 3$, a contradiction. Thus, we have $c_{e_1} > 0$, which then implies that $c_{e_1} = c_{e_{2l+1}} = c_{f_1'} = 1$. By Remark \ref{rmk.coefficient}, we deduce that $c_{f_j} = 0$ if $j$ is odd and $c_{f_j} = 2$ if $j$ is even, and $c_{f_j'} = 1$ for all $1 \le j \le q$.

If $p$ is odd then, by a similar argument as above, we have the facet
\begin{align*}
G_{12} = & \{e_j ~|~ 1 \le j \text{ odd } \le 2l+1 \} \cup \{f_j ~|~ 1 \le j \text{ even } \le p\} \\
& \cup \{f_j' ~|~ 1\le j \le q\} \cup \{e_j' ~|~ 1 \le j \text{ odd } \leq 2l'+1\},
\end{align*}
and if $p$ is even then we get the facet
\begin{align*}
G_{12}' = & \{e_j ~|~ 1 \le j \text{ odd } \le 2l+1 \} \cup \{f_j ~|~ 1 \le j \text{ even } \le p\} \\
& \cup \{f_j' ~|~ 1\le j \le q\} \cup \{e_j' ~|~ 1 \le j \text{ even } \leq  2l'+1\}.
\end{align*}

\begin{figure}[h!]
\centering
\includegraphics[height=2in]{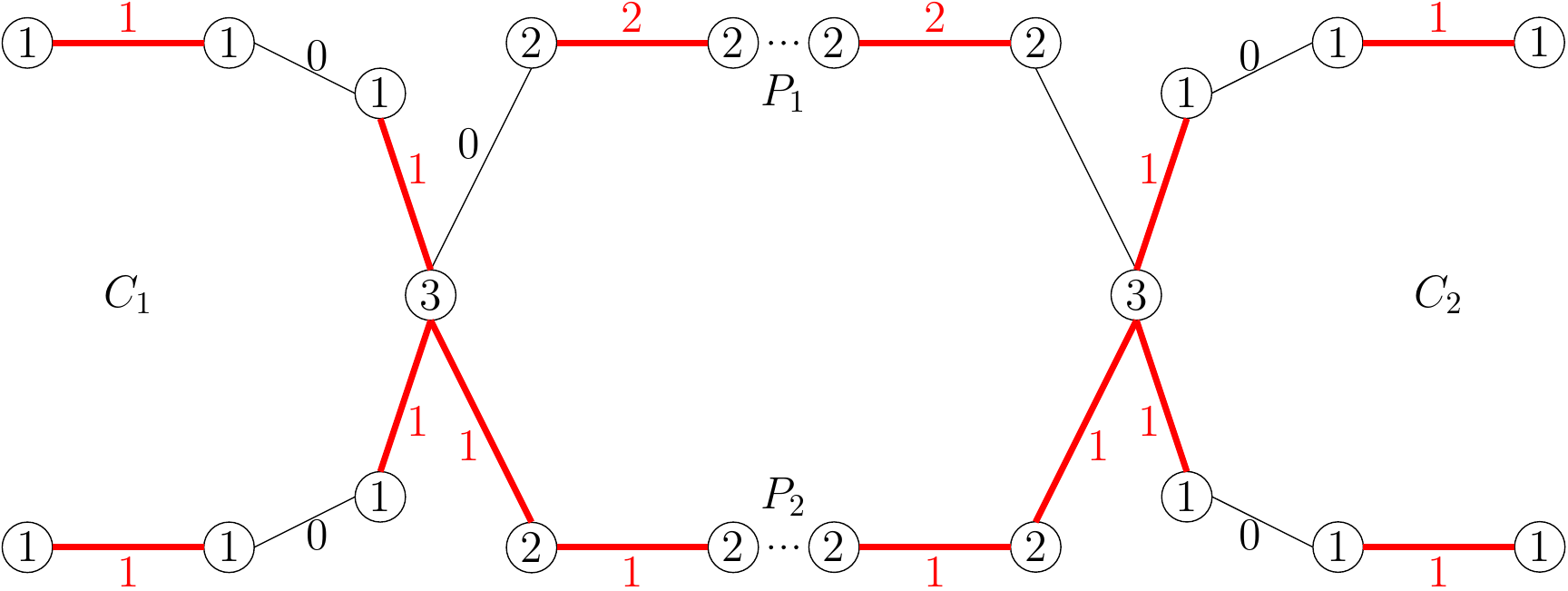}
\centering\caption{A facet of $\Delta_\s$: Case 1c when $p$ is odd.}\label{fig.1c}
\end{figure}

Hence, in Case 1, we conclude that $\Delta_\s$ has 4 facets $F_{11} \cup F_{12} \cup F_{21} \cup F_{22}$, where $F_{ij} = G_{ij}$ or $G_{ij}'$.

\noindent{\bf Case 2:} $P_1$ and $P_2$ share exactly one endpoint vertex. Without loss of generality, we may assume that $z_0 = w_0 = x_1$, $z_p = y_1$ and $z_q = y_v$, where $v \not= 1$. Noting that $c_{f_1}, c_{f_1'} \le 2$, our arguments proceed similarly to that in Case 1 by considering subcases depending on the value of $c_{f_1}$. In each subcase, we shall point out the similarities with Case 1, and leave the details to the interested reader.

\noindent{\it Case 2a:} $c_{f_1} = 2$. By similar arguments as that of Case 1a, we have $c_{f_j} = 2$ if $j$ is odd and $c_{f_j} = 0$ if $j$ is even, $c_{e_j} = 0$ for $j$ odd and $c_{e_j} = 1$ for $j$ even, and $c_{f_j'} = 1$ for all $j = 1, \dots, q$. Note also that since $c_{f_q'} = 1$ and $s_{y_v} = 2$, we have $c_{e_{v-1}'} + c_{e_v'} = 1$. Thus, again by considering of $p$ is odd or even (i.e., if $c_{f_p}$ is 2 or 0) and making use of Remark \ref{rmk.coefficient}, we deduce that either
$$c_{e_j'} = \left\{ \begin{array}{lcl} 1 & \text{if} & j \text{ is even} \\ 0 & \text{if} & j \text{ is odd} \end{array} \right. \text{ or }
c_{e_j'} = \left\{ \begin{array}{lcl} 0 & \text{if} & j \text{ is even} \\ 1 & \text{if} & j \text{ is odd.} \end{array} \right.$$
Therefore, either $p$ is odd or even, we always obtain a facet for $\Delta_\s$.
\begin{figure}[h!]
\centering
\includegraphics[height=2in]{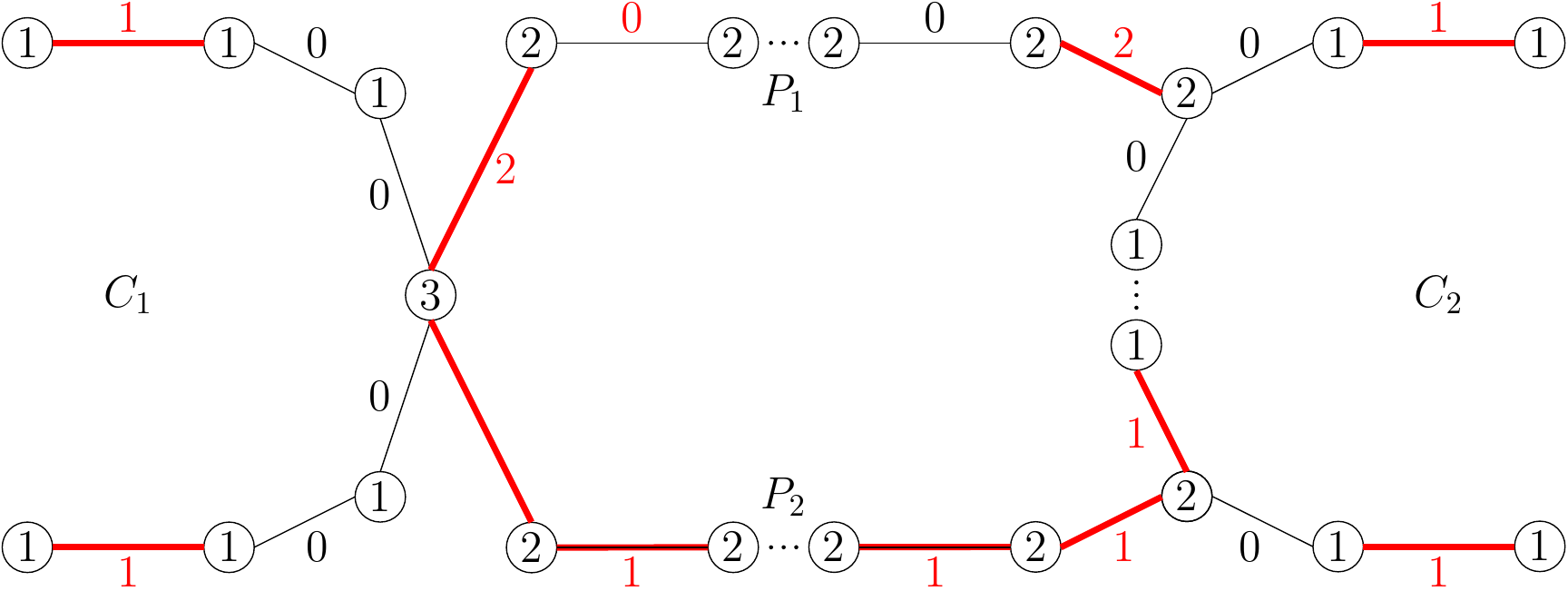}
\centering\caption{A facet of $\Delta_\s$: Case 2a when $p$ odd and $v$ is odd.}\label{fig.2a}
\end{figure}

\noindent{\it Case 2b:} By the same arguments as that of Case 1b, we have $c_{f_j} = 1$ for all $j = 1, \dots, p$. Suppose first that $c_{e_1} > 0$. Then, as before, we deduce that $c_{e_j} = 1$ if $j$ is odd and $c_{e_j} = 0$ if $j$ is even, and that $c_{f_j'} = 0$ if $j$ is odd and $c_{f_j'} = 2$ if $j$ is even. Now, by considering if $q$ is odd or even and considering the value $s_{y_1}=2$, we obtain a facet of $\Delta_\s$.

Suppose, on the other hand, that $c_{e_1} = 0$. The same arguments as in Case 1b then imply that $c_{e_j} = 0$ if $j$ is odd and $c_{e_j} = 1$ if $j$ is even, and $c_{f_j'} = 2$ if $j$ is odd and $c_{f_j'} = 0$ if $j$ is even. Once again, by considering if $q$ is odd or even and the value $s_{y_1}=2$, we always obtain another facet of $\Delta_\s$.

\noindent{\it Case 2c:} $c_{f_1} = 0$. As in Case 1c, we deduce that $c_{f_j} = 0$ if $j$ is odd and $c_{f_j} = 2$ if $j$ is even, and $c_{f_j'} = 1$ for all $1 \le j \le q$. To this end, by considering if $p$ is odd or even and the value $s_{y_v} = 2$, we again obtain a facet of $\Delta_\s$.

Hence, in Case 2, we obtain 4 facets for $\Delta_\s$ (one for each Case 2a and Case 2c, and two for Case 2b).

\noindent{\bf Case 3:} $P_1$ and $P_2$ do not share any endpoint vertices. Without loss of generality (and after a relabeling if necessary), we may assume that $z_0 = x_1$, $w_0 = x_u$, $z_p = y_1$, and $w_q = y_v$, where $u \not= 1$ and $v \not= 1$. Our arguments will proceed similarly to that in Cases 1 and 2 with some minor differences, which we shall identify in detail. As before, noting that $c_{f_1}, c_{f_1'} \le 2$, we shall consider subcases depending on the value of $c_{f_1}$.

\noindent{\it Case 3a:} $c_{f_1} = 2$. Since $s_{x_1} = 2$ in this case, we must have $c_{e_1} = 0$. It then follows from Remark \ref{rmk.coefficient} that $c_{e_j} = 0$ if $j$ is odd and $c_{e_j} = 1$ if $j$ is even. Particularly, at $x_u$, we always have $c_{e_{u-1}} + c_{e_u} = 1$. Since $s_{x_u} = 2$, this implies that $c_{f_1'} = 1$. Applying Remark \ref{rmk.coefficient} to the paths $P_1$ and $P_2$, we deduce that $c_{f_j} = 2$ if $j$ is odd, $c_{f_j} = 0$ if $j$ is even, and $c_{f_j'} = 1$ for all $1 \le j \le q$.

To this end, we again consider if $p$ is odd or even. If $p$ is odd then $c_{f_p} = 2$, which forces $c_{e_1'} = c_{e_{2l'+1}'} = 0$. Also, at $y_v$, since $c_{f_q'} =1$, we have $c_{e_{v-1}'} + c_{e_v'} = 1$. It follows, by Remark \ref{rmk.coefficient}, that $c_{e_j'} = 0$ if $j$ is odd and $c_{e_j'} = 1$ if $j$ is even. Thus, in this case we obtain a facet of $\Delta_\s$. On the other hand, if $p$ is even then $c_{f_p} = 0$, which forces $c_{e_1'} = c_{e_{2l'+1}'} = 1$. This, together with the hypothesis that $s_{y_v} = 2$ and Remark \ref{rmk.coefficient}, implies that $c_{e_j'} = 1$ if $j$ is odd and $c_{e_j'} = 0$ if $j$ is even. We again obtain a facet of $\Delta_\s$.
\begin{figure}[h!]
\centering
\includegraphics[height=2in]{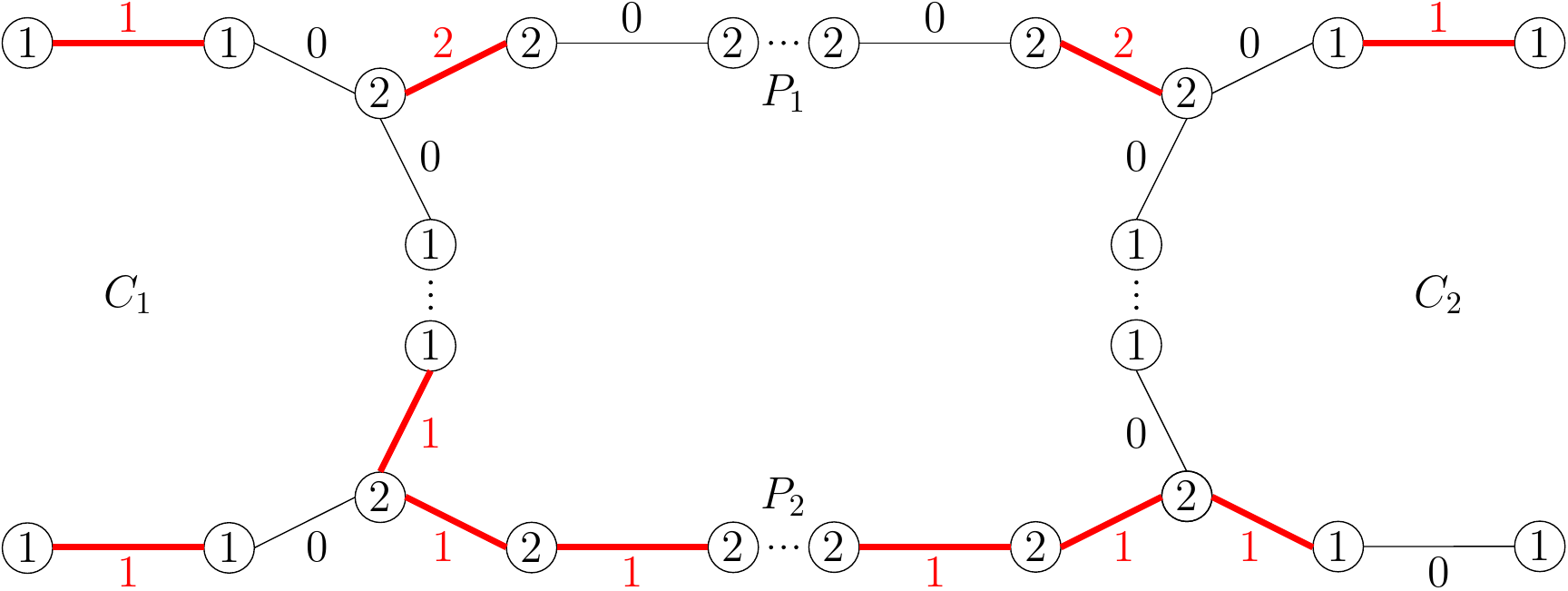}
\centering\caption{A facet of $\Delta_\s$: Case 3a when $p$ is odd, $u$ is odd and $v$ is even.}\label{fig.3a}
\end{figure}

\noindent{\it Case 3b:} $c_{f_1} = 1$. Since $s_{x_1} = 2$, we must have either $c_{e_1} = 1$ and $c_{e_{2l+1}} = 0$ or $c_{e_1} = 0$ and $c_{e_{2l+1}} = 1$. By making use of Remark \ref{rmk.coefficient} and tracing around the odd cycle $C_1$, it can be seen that the coefficients $c_{e_j}$'s alternate between 0 and 1, except at exactly one place, where the $c_{e_j}$'s remain 1 and 1 or 0 and 0. It follows from the hypothesis that this exception must be at $x_u$, where we have $s_{x_u} = 2$.

Suppose that $c_{e_{u-1}} = c_{e_u} = 0$. Then since $s_{x_u} = 2$, we must have $c_{f_1'} = 2$. Relabel the vertices on $C_1$ starting with $x_u$ (i.e., $x_u$ is now labeled by $x_1$), and interchange the role of $P_1$ and $P_2$. Then we are back to the same situation as in Case 3a, which again gives us a facet of $\Delta_\s$.

Suppose, on the other hand, that $c_{e_{u-1}} = c_{e_u} = 1$. Then we have $c_{f_1'} = 0$. By the same relabeling of the vertices and interchanging the role of $P_1$ and $P_2$, we bring to the situation of Case 3c, which we shall consider next. This gives us another facet of $\Delta_\s$ for Case 3b.

\noindent{\it Case 3c:} $c_{f_1} = 0$. By Remark \ref{rmk.coefficient}, we have $c_{f_j} = 0$ if $j$ is odd and $c_{f_j} = 2$ if $j$ is even. Since $s_{x_1} = 2$, it follows that $c_{e_1} = c_{e_{2l+1}} = 1$. By Remark \ref{rmk.coefficient} and tracing around $C_1$, we have $c_{e_j} = 1$ if $j$ is odd and $c_{e_j} = 0$ if $j$ is even. Particularly, at $x_u$, we have $c_{e_{u-1}}=1$ and  $c_{e_u} = 0$ or $c_{e_{u-1}}=0$ and  $c_{e_u} = 1$. This implies that $c_{f_1'} = 1$. It follows from Remark \ref{rmk.coefficient} again that $c_{f_j'} = 1$ for all $j = 1, \dots, q$.

To this end, we once again consider if $p$ is odd or even. Similarly to the arguments in Case 3a, we again obtain a facet for $\Delta_\s$.

Hence, in Case 3, we show that $\Delta_\s$ contains 4 facets (one for each Case 3a and Case 3c, and two for Case 3b).
\end{proof}


\section{Non-Cohen-Macaulay toric rings} \label{sec.nCM}

In this section, we shall use the forbidden structure described in the last section to give a sufficient condition for the toric ring $k[G]$ of a graph not to be Cohen-Macaulay.

Observe that a simplicial complex is determined by its facets, so we shall denote the simplicial complex with facets $F_1, \dots, F_t$ by $\langle F_1, \dots, F_t\rangle$. The next main result of our paper is stated as follows.

\begin{theorem} \label{thm.nCM}
Let $G = (V,E)$ be a simple graph. Suppose that $|E| \le |V|+2$ and $G$ contains an induced subgraph which consists of:
\begin{itemize}
\item two vertex-disjoint odd cycles; and
\item two vertex-disjoint (except possibly at their endpoint vertices) paths of length $\ge 2$ connecting these cycles.
\end{itemize}
Then the toric ring $k[G]$ is not Cohen-Macaulay.
\end{theorem}

\begin{proof} By Theorem \ref{thm.reduction}, we may assume that $G$ is the graph consisting of two vertex-disjoint odd cycles connected by exactly two vertex-disjoint (except possibly at their endpoints) paths of length $\ge 2$, and we may relabel the vertices and edges of $G$ as in Notation \ref{not.basic}.

It follows from \cite[Proposition 3.2]{Vill} that $\dim k[G] = |V|$. By the Auslander-Buchsbaum formula, we also have
$$\depth k[G] + \pd k[G] = \depth S = |E| = |V|+2.$$
To show that $k[G]$ is not Cohen-Macaulay, it suffices to show that $\pd k[G] \ge 3$. To achieve this, we shall pick a particular multidegree $\s$ and prove that $\beta_{3,\s}(k[G]) \not= 0$. Indeed, take $\s = (s_x ~|~ x \in V)$, where $s_x = 1 + \big|\{ j ~|~ x \in P_j\}\big|.$ By Theorem \ref{thm.homology}, we have
$$\beta_{3,\s} = \dim_k \tilde{H}_2(\Delta_\s;k).$$
Thus, it remains to show that $\dim_k \tilde{H}_2(\Delta_\s;k) \not= 0$.

Based on our arguments in Proposition \ref{lem.facets}, we shall consider the following cases:
\begin{itemize}
\item $P_1$ and $P_2$ share both endpoint vertices;
\item $P_1$ and $P_2$ share only one endpoint vertex; and
\item $P_1$ and $P_2$ share no endpoint vertices.
\end{itemize}
Our arguments for these cases are basically identical since in each case, as shown in Proposition \ref{lem.facets}, $\Delta_\s$ has exactly 4 facets which can be explicitly described. For this reason, we shall only present the detailed proof for the first case, where $P_1$ and $P_2$ share both endpoint vertices.

By the proof of Proposition \ref{lem.facets}, $\Delta_\s$ consists of 4 facets $F_{11}, F_{12}, F_{21}$ and $F_{22}$ as described. For the simplicity of arguments, we shall also assume that $p$ and $q$ are odd (the other situations follow exactly the same arguments). That is, $F_{11}, F_{12}, F_{21}$ and $F_{22}$ are as depicted in Figures \ref{fig.1a}, \ref{fig.1c}, \ref{fig.1b1}, and \ref{fig.1b0}. Particularly, we have
\begin{align*}
F_{11} = & \{e_j ~|~ 1 \le j \text{ odd } \le 2l+1 \} \cup \{f_j ~|~ 1 \le j \text{ odd } \le p\} \\
& \cup \{f_j' ~|~ 1\le j \le q\} \cup \{e_j' ~|~ 1 \le j \text{ even } 2l'+1\},\\
F_{12} = & \{e_j ~|~ 1 \le j \text{ even } \le 2l+1 \} \cup \{f_j ~|~ 1 \le j \text{ even } \le p\} \\
& \cup \{f_j' ~|~ 1\le j \le q\} \cup \{e_j' ~|~ 1 \le j \text{ odd } 2l'+1\},\\
F_{21} = & \{e_j ~|~ 1 \le j \text{ odd } \le 2l+1 \} \cup \{f_j ~|~ 1 \le j \le p\} \\
& \cup \{f_j' ~|~ 1\le j \text{ even }\le q\} \cup \{e_j' ~|~ 1 \le j \text{ odd } 2l'+1\},\\
F_{22} = & \{e_j ~|~ 1 \le j \text{ even } \le 2l+1 \} \cup \{f_j ~|~ 1 \le j \le p\} \\
& \cup \{f_j' ~|~ 1\le j \text{ odd }\le q\} \cup \{e_j' ~|~ 1 \le j \text{ even } 2l'+1\}.
\end{align*}

For $i = 1,2$, let $\Delta_i = \langle F_{i1}, F_{i2}\rangle$. Then $\Delta_\s = \Delta_1 \cup \Delta_2$. Consider the nested Mayer-Vietoris sequences as in Figure \ref{fig.MV}.
\begin{figure}[h!]
\[\xymatrix{
    & & & {\vdots}\\
    & & & {\begin{array}{c}\tH_0(\langle F_{1,1}\rangle \cap\Delta_2)\\
            \oplus\\
            \tH_0(\langle F_{1,2}\rangle \cap\Delta_2)\end{array}} \ar[u]\\
    & & &  {\tH_0(\langle F_{1,1}\rangle \cap \langle F_{1,2}\rangle\cap\Delta_2)} \ar[u]\\
    {\cdots} \ar[r] & {\begin{array}{c}\tH_2(\Delta_1)\\ \oplus\\ \tH_2(\Delta_2)\end{array}} \ar[r] &
            {\tH_2(\Delta_\s)}  \ar[r] &
            {\tH_1(\Delta_1\cap\Delta_2)} \ar[u]  \ar[r] &
            {\begin{array}{c}\tH_1(\Delta_1)\\ \oplus\\ \tH_1(\Delta_2)\end{array}} \ar[r] & \cdots \\
    & & & {\begin{array}{c}\tH_1(\langle F_{1,1}\rangle \cap\Delta_2)\\
            \oplus\\
            \tH_1(\langle F_{1,2}\rangle \cap\Delta_2)\end{array}} \ar[u]\\
    & & & {\vdots} \ar[u]
}\]
\centering\caption{Nested Mayer-Vietoris sequences.}\label{fig.MV}
\end{figure}
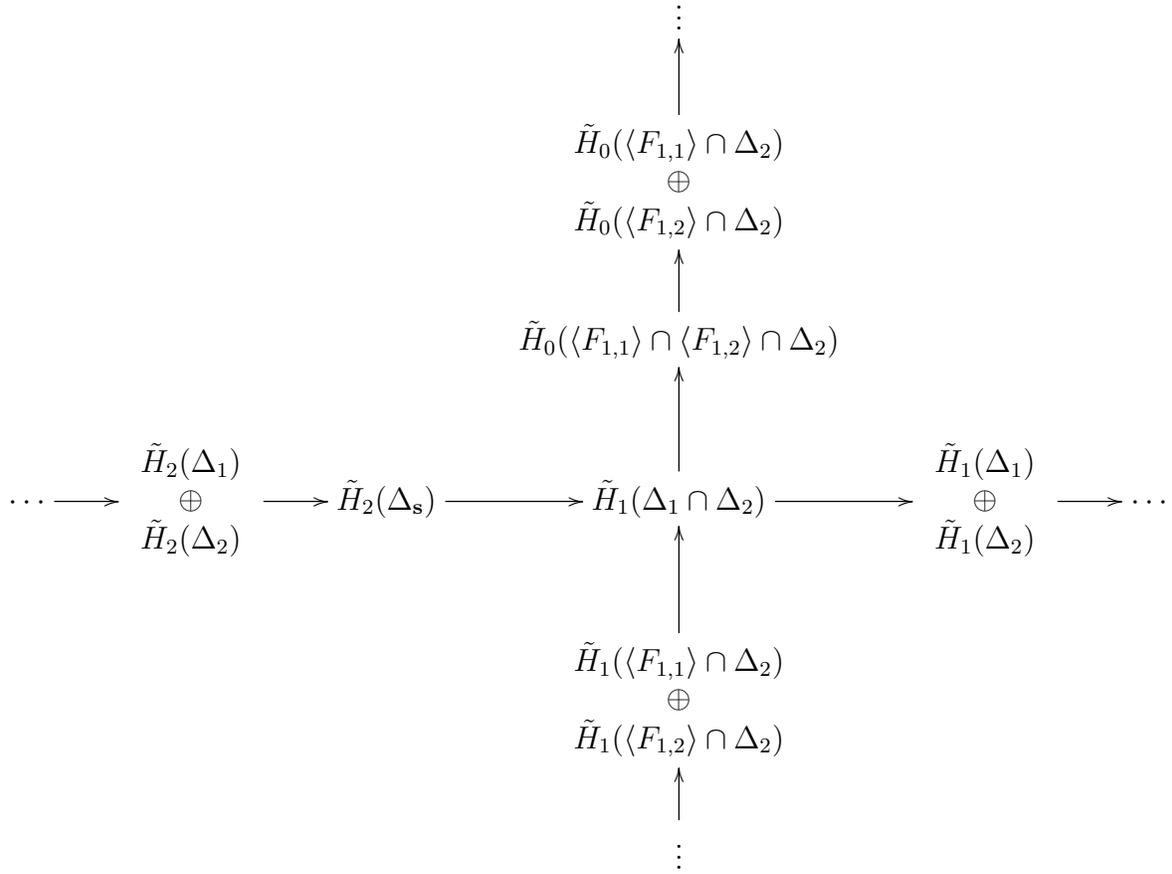

It can be seen that $F_{11} \cap F_{12} = \{f_j' ~|~ 1 \le j \le q\}$ and $F_{21} \cap F_{22} = \{f_j ~|~ 1 \le j \le p\}$. Thus, $\Delta_i$, for $i=1,2$, consists of two facets with nonempty intersection and, therefore, is contractible and has trivial reduced homology. That is, $\tH_l(\Delta_i)={0}$ for all $i=1,2$ and $l \ge 0$. It follows from the Mayer-Vietoris sequence that $\tH_2(\Delta_\s)\cong\tH_1(\Delta_1\cap\Delta_2)$.

Observe further that, for $j = 1,2$, $\langle F_{1j}\rangle \cap \Delta_2$ also consists of two facets with nonempty intersection and, thus, is contractible. Therefore, $\tH_l(\langle F_{1,j}\rangle \cap\Delta_2) = {0}$ for all $j=1,2$ and $l \ge 0$. Hence, we have the following isomorphisms
\[
\tH_2(\Delta_\s)\cong\tH_1(\Delta_1\cap\Delta_2)\cong\tH_0(\langle F_{1,1}\rangle \cap \langle F_{1,2}\rangle \cap \Delta_2).
\]

Observe finally that
\begin{align*}
\langle F_{1,1}\rangle \cap \langle F_{1,2}\rangle \cap \Delta_2 & = \langle\{f_j' ~|~ 1 \le j \le q\}\rangle \cap \langle F_{21} \cup F_{22}\rangle \\
& = \langle \{f_j' ~|~ 1 \le j \text{ even } \le q\}\rangle \cup \langle \{f_j' ~|~ 1 \le j \text{ odd } \le q\}\rangle.
\end{align*}
is the simplicial complex with exactly two disjoint facets. It follows that
$$\dim_k\tH_0(\langle F_{1,1}\rangle \cap \langle F_{1,2}\rangle \cap\Delta_2)=1,$$
which completes the proof of the theorem.
\end{proof}

When the two paths connecting the induced odd cycles $P_1$ and $P_2$ are allowed to share internal vertices, we no longer necessarily have a forbidden structure. This is illustrated in the following examples.

\begin{figure}[h!]
\centering
\includegraphics[height=2in]{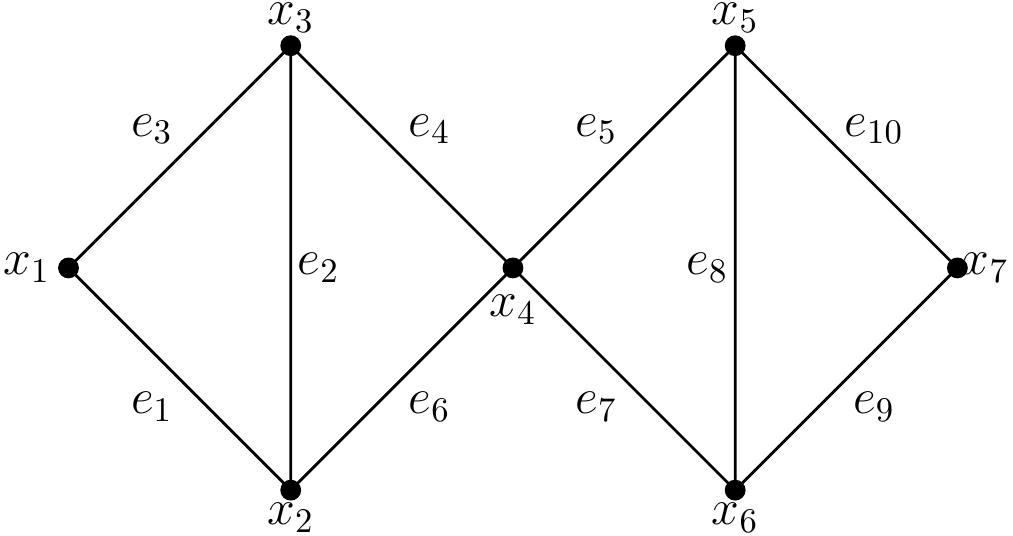}
\centering\caption{When paths $P_1$ and $P_2$ have common internal vertices and $k[G]$ is Cohen-Macaulay.}\label{fig.CMgraph}
\end{figure}
\begin{example} \label{ex.1}
Let $G$ be the graph in Figure \ref{fig.CMgraph}. By a direct computation, we have $\depth k[G] = \dim k[G] = 7$. Thus, $k[G]$ is Cohen-Macaulay.
\end{example}

\begin{figure}[h!]
\centering
\includegraphics[height=2in]{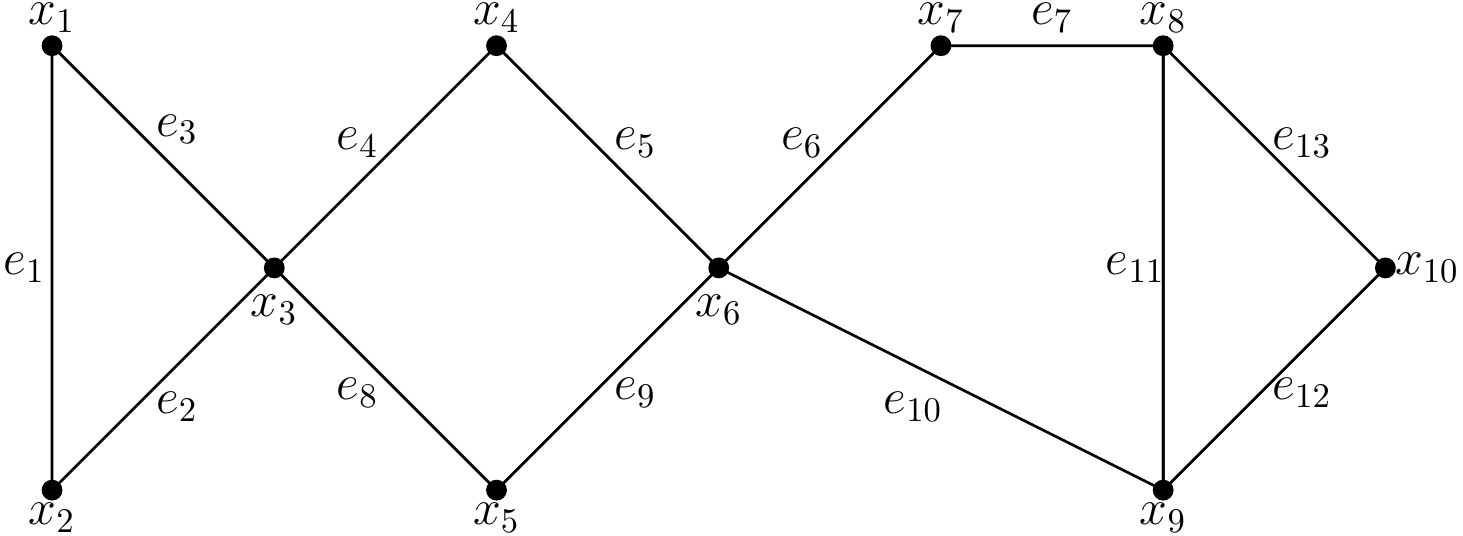}
\centering\caption{When paths $P_1$ and $P_2$ have common internal vertices and $k[G]$ is not Cohen-Macaulay.}\label{fig.nCMgraph}
\end{figure}
\begin{example} \label{ex.2}
Let $G$ be the graph in Figure \ref{fig.nCMgraph}. By a direct computation, we have $\depth k[G] = 8 < \dim k[G] = 9$. Particularly, $k[G]$ is not Cohen-Macaulay.
\end{example}

Theorem \ref{thm.nCM} is no longer true without the condition that $|E| \le |V|+2$. We thank K. Kimura for showing us the following example.

\begin{figure}[h!]
\centering
\includegraphics[height=2in]{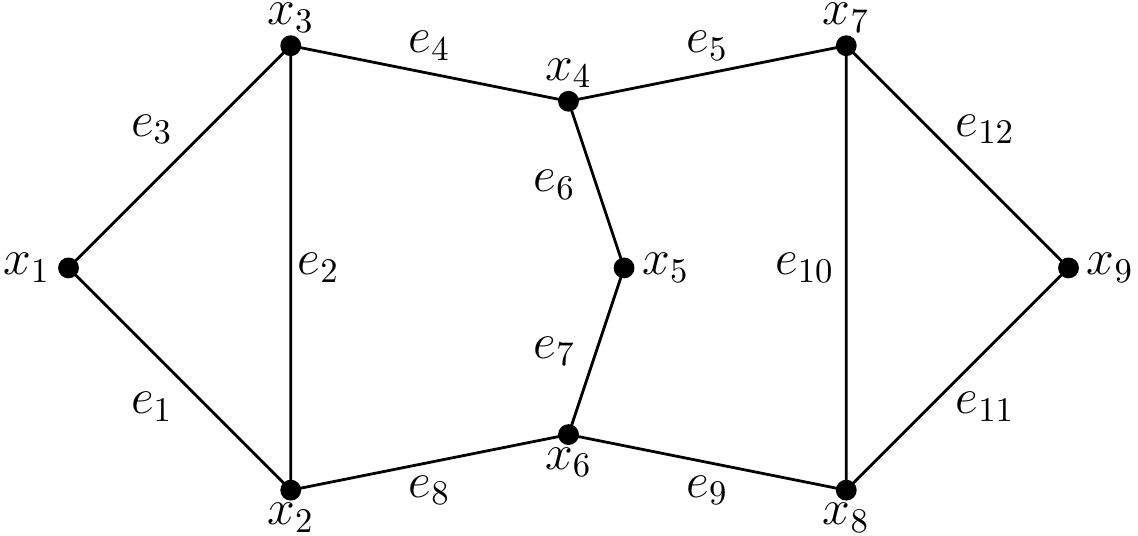}
\centering\caption{Graph with Cohen-Macaulay toric ring and an induced subgraph with non-Cohen-Macaulay toric ring.}\label{fig.Kimura}
\end{figure}
\begin{example} \label{ex.Kimura}
Let $G$ be the graph in Figure \ref{fig.Kimura} and let $H$ be the induced subgraph of $G$ which excludes the vertex $x_9$. Then $|E_G| = |V_G| + 3.$ It can be seen that $H$ is a forbidden structure, and $k[H]$ is not Cohen-Macaulay with $\depth k[H] = 7 < \dim k[H] = 8$. However, $k[G]$ is Cohen-Macaulay with $\depth k[G] = \dim k[G] = 9.$
\end{example}

Our computation in the proof of Theorem \ref{thm.nCM} further gives us the following bound for the regularity of $k[G]$ in terms of the size of its forbidden structure.

\begin{corollary} \label{cor.regforbidden}
Let $G = (V,E)$ be a simple graph that contains an induced subgraph $H$ which consists of:
\begin{itemize}
\item two vertex-disjoint odd cycles; and
\item two vertex-disjoint (except possibly at their endpoint vertices) paths of length $\ge 2$ connecting these cycles.
\end{itemize}
Suppose that $|V_H| = t$ and the lengths of the two paths connecting the two odd cycles in $H$ are $p$ and $q$. Then $\reg k[G] \ge t+p+q - 1$.
\end{corollary}

\begin{proof} It is easy to see that with the choice $\s = (s_x ~|~ x \in V_H)$, where
$$s_x = 1 + \big|\{\text{paths connecting odd cycles in $H$ that pass through $x$}\}\big|,$$
as given in the proof of Theorem \ref{thm.nCM}, we have
$|\s| = |V_H| + p + q + 2 = t+p+q+2.$
By Corollary \ref{cor.reduction} and the proof of Theorem \ref{thm.nCM}, we also have that
$$\beta_{3,\s}(k[G]) \ge \beta_{3,\s}(k[H]) > 0.$$
Thus, $\reg k[G] \ge |\s| - 3 = t+p+q-1.$
\end{proof}



\begin{thebibliography}{999}
\bibitem{AH} A. Aramova and J. Herzog, Koszul cycles and Eliahou-Kervaire type resolutions. J. Algebra \textbf{181} (1996), 347-370.
\bibitem{BM} C. Bertone, and V. Micale, On the dimension of the minimal vertex cover semigroup ring of an unmixed bipartite graph. Matematiche (Catania) \textbf{63} (2008), no. 2, 157-163 (2009).
\bibitem{BOVT} J. Biermann, A. O'Keefe and A. Van Tuyl, Bounds on the regularity of toric ideals of graphs. Adv. in Appl. Math. \textbf{85} (2017), 84-102.
\bibitem{BS}  M P.Brodmann, R.Y. Sharp, Local Cohomology: An Algebraic Introduction with Geometric Applications. Cambridge University Press, (2012).
\bibitem{CN} A. Corse and U. Nagel, Monomial and toric ideals associated to Ferrers graphs. Trans. Amer. Math. Soc. \textbf{361} (2009), no. 3, 1371-1395.
\bibitem{CM} D. Cox and E. Materov, Tate resolutions for Segre embeddings. Algebra Number Theory \textbf{2} (2008), no. 5, 523-550.
\bibitem{DA} A. D'Al\`i, Toric ideals associated with gap-free graphs. J. Pure Appl. Algebra \textbf{219} (2015), no. 9, 3862-3872.
\bibitem{CS} D. Cox, J.B. Little, and H.K. Schenck, Toric varieties. Graduate Studies in Mathematics, \textbf{124}. American Mathematical Society, Providence, RI, 2011. xxiv+841 pp.
\bibitem{D} R. Diestel, Graph theory. Fourth edition. Graduate Texts in Mathematics, \textbf{173}. Springer, Heidelberg, 2010. xviii+437 pp.
\bibitem{DG} L.R. Doering and T. Gunston, Algebras arising from bipartite planar graphs. Comm. Algebra \textbf{24} (1996), no. 11, 3589-3598.
\bibitem{E} D. Eisenbud, Commutative algebra. With a view toward algebraic geometry. Graduate Texts in Mathematics, \textbf{150}. Springer-Verlag, New York, 1995. xvi+785 pp.
\bibitem{GV} I. Gitler and C.E. Valencia, Multiplicities of edge subrings. Discrete Math. \textbf{302} (2005), no. 1-3, 107-123.
\bibitem{HHKO} T. Hibi, A. Higashitani, K. Kimura, and A. O'Keefe, Depth of edge rings arising from finite graphs. Proc. Amer. Math. Soc. \textbf{139} (2011), no. 11, 3807-3813.
\bibitem{HMO} T. Hibi, K. Matsuda, and H. Ohsugi, Strongly Koszul edge rings. Acta Math. Vietnam. \textbf{41} (2016), no. 1, 69-76.
\bibitem{H} M. Hochster,  Rings of invariants of tori, cohen-macaulay rings generated by monomials, and polytopes. Ann. of Math. \textbf{96} (1972), 318-337.
\bibitem{OH} H. Ohsugi and T. Hibi, Toric ideals generated by quadratic binomials. J. Algebra \textbf{218} (1999), 509-527.
\bibitem{OHH} H. Ohsugi, J. Herzog, and T. Hibi, Combinatorial pure subrings. Osaka J. Math. \textbf{37} (2000), 745-757.
\bibitem{SV} J. St\"uckrad and W. Vogel, On Segre products and applications. J. Algebra \textbf{54} (1978), 374-389.
\bibitem{RTT} E. Reyes, C. Tatakis, and A. Thoma, Minimal generators of toric ideals of graphs. Adv. in Appl. Math. \textbf{48} (2012), no. 1, 64-78.
\bibitem{TT} C. Tatakis and A. Thoma, On the universal Gröbner bases of toric ideals of graphs. J. Combin. Theory Ser. A \textbf{118} (2011), no. 5, 1540-1548.
\bibitem{Vill} R.H. Villarreal, Rees algebras of edge ideals. Comm. in Algebra \textbf{23} (1995), 3513-3524.

\end{thebibliography}
\end{document}